\documentclass[a4paper,12pt,reqno]{amsart}
\usepackage{mathrsfs,amsmath,amssymb,latexsym,amsfonts, amscd}

\title{Some birationality criteria on 3-folds with $p_g>1$}
\author{Meng Chen}

\dedicatory{Dedicated to the memory of professor Gang Xiao}

\address{\rm School of Mathematical Sciences, Fudan University, Shanghai 200433, China}
\email{mchen@fudan.edu.cn}

\thanks{Supported by National Natural Science Foundation of China (\#11171068, \#11121101, \#11231003)}

\newcommand{\bC}{{\mathbb C}}
\newcommand{\bQ}{{\mathbb Q}}
\newcommand{\bP}{{\mathbb P}}
\newcommand{\roundup}[1]{\lceil{#1}\rceil}

\newcommand\Vol{\text{\rm Vol}}
\newcommand\lrw{\longrightarrow}
\newcommand\rw{\rightarrow}
\newcommand\hrw{\hookrightarrow}

\newcommand\OO{{\mathcal{O}}}
\newcommand\OX{{\mathcal{O}_X}}

\newcommand\He{H_{\varepsilon}}
\newcommand\eps{\varepsilon}
\newcommand{\map}{\varphi}
\newcommand{\genum}{\geq_{\text{num}}}
\newtheorem{thm}{Theorem}[section]
\newtheorem{lem}[thm]{Lemma}

\theoremstyle{definition}
\newtheorem{defn}[thm]{Definition}

\newtheorem{fact}[thm]{Fact}

\newtheorem{rem}[thm]{Remark}
\theoremstyle{remark}

\newtheorem{Rema}{\bf Remark}
\begin{document}
\begin{abstract} We give some birationality criteria for $\map_m$ ($m=4$, $5$, $6$, $7$) on general type 3-folds with $p_g\geq 2$ by means of an intensive classification. In particular, we show that $\varphi_7$ is not birational if and only if $p_g(X)=2$ and $X$ admits a genus 2 curve family of canonical degree $\frac{2}{3}$.  When the canonical volume is large, we also characterize the birationality of $\varphi_4$, $\varphi_5$ and $\varphi_6$. 
\end{abstract}
\maketitle


\pagestyle{myheadings}
\markboth{\hfill M. Chen\hfill}{\hfill Some birationality criteria on 3-folds\hfill}
\numberwithin{equation}{section}
\section{\bf Introduction}
We work over any algebraically closed field $k$ of characteristic 0 (for instance, $k=\bC$).

Pluricanonical maps are usually important tools to study birational geometry of projective varieties. Recently, due to the boundedness theorem independently proved by Hacon-M$^{\text{c}}$kernan, Takayama and Tsuji, it has raised a hope to look into explicit birational geometry of high dimensional varieties of general type. In dimension 3, the development is much favorable by virtue of \cite{Ex1, Ex2} where the following is known:
\begin{itemize}
\item[$\diamond$] the volume $\Vol(V)\geq \frac{1}{2660}$, and
\item[$\diamond$] the pluricanonical map $\map_m$ is birational for all $m\geq 73$
\end{itemize}
where $V$ is any nonsingular projective 3-fold of general type. Even though, birational geometry in dimension 3 is far from being well-understood. 

As far as we know 3-folds with very small volume and very bad pluricanonical behaviors all have invariants $p_g=q=0$ and they correspond to surfaces with $p_g=q=0$ (i.e. Godeaux surfaces, Campedelli surfaces, Burniat surfaces and so on). Threefolds with $p_g=1$ form a very typical class of which pluricanonical behaviors are slightly better. Those with $p_g>1$ should be ragarded as general objects from the point of view of ``moduli''. A feasible strategy to study 3-folds of general type might be to distinguish the set of 3-folds into 3 subsets and to treat each of them by an appropriate method, say,
$${\mathfrak V}_3=\underset{(={\mathfrak V}_{3,0})}{\{X_1| p_g(X_1)=0\}}\cup \underset{(={\mathfrak V}_{3,1})}{\{X_2| p_g(X_2)=1\}}\cup \underset{(={\mathfrak V}_{3,2})}{
\{X_3| p_g(X_3)\geq 2\}}$$
where $X_i$ denotes an arbitrary 3-fold of general type. 
Though all above 3 parts are known up to some extent, none of them is clear enough. The motivation of this article is to go further on classification. 

In this paper we are interested in the third part ${\mathfrak V}_{3,2}$. Based on our previous papers \cite{IJM, MA} and Chen-Zhang \cite{MZ}, we would like to investigate certain parallel phenomenon between surfaces and 3-folds. {}First of all let us make the following interesting comparison of known results: \bigskip

\begin{center}
\noindent{\tiny
\begin{tabular}{c|c}
 {\bf Birationality of $\phi_m$ on surfaces $S$} & {\bf Birationality of $\map_m$ on 3-folds $X$}\\
 & {\bf with $p_g\geq 2$}\\
(Bombieri \cite{Bom}, Miyaoka \cite{Miy})& (Chen \cite{IJM, MA}, Chen-Zhang \cite{MZ})\\
\hline\hline
&\\
$\phi_5$ is birational  &
$\map_8$ is birational \\
& $\map_7$ (?)\\
&\\
\hline
&\\
$(K^2,p_g)\neq (1,2)\iff$  & $p_g\geq 3\Longrightarrow$ $\map_6$ is birational;\\
$\phi_4$ is birational & $p_g=2$ (?)\\
&\\
\hline
&\\
$(K^2, p_g)\neq (1,2)$, $(2,3)$ $\iff$& $p_g\geq 4\Longrightarrow$ $\map_5$ is birational;\\
$\phi_3$ is birational & $p_g=2,\ 3$ (?)\\
&\\
\hline
&\\
when $K^2\geq 10$ or $p_g\geq 6$, & when $p_g\geq 5$,\\
$\phi_2$ is non-birational {\it iff} $S$ admits & $\map_4$ is non-birational {\it iff} $X$ admits\\
a family of genus 2 curves; &
a genus 2 curve family of\\
 &canonical degree 1\\
& \\
when $K^2\leq 9$, $\phi_2$ (?) &  when $p_g\leq 4$, $\map_4$ (?)\\
&\\
\hline
$\phi_1$ (?)& $\map_1$, $\map_2$, $\map_3$ (?)\\
\hline\hline
\end{tabular}}\end{center}
where ``?'' means an open status, $\phi_n:=\Phi_{|mK_S|}$ and $\map_m:=\Phi_{|mK_X|}$.

We start with a translation of Bombieri's famous theorem (\cite{Bom}) from another point of view.
\medskip

\noindent{\bf Theorem 0.} {\em Let $S$ be a minimal projective surface of general type. Then $\map_4$ is non-birational if and only if $S$ admits a genus 2 curve family ${\mathscr C}$ of canonical degree $1$.}
\begin{proof} For a general curve $C\in {\mathscr C}$ with $(K_S\cdot C)=1$, we must have $C^2>0$ since it is an odd number and $C$ is moving. Then we get $K_S^2=1$ by the Hodge index theorem. The Noether inequality implies $p_g(S)=0,\ 1,\ 2$. The case $p_g(S)=0$ is impossible since $(K_S\cdot C)\geq 2$ by Miyaoka \cite[Lemma 5]{Miy}. The case $p_g(S)=1$ is impossible either since, otherwise, $K_S\equiv C$ which contradicts to the fact that $S$ is simply connected (see Bombieri \cite{Bom}). Thus $S$ is a $(1,2)$ surface and $\map_4$ is non-birational according to Bombieri \cite{Bom}.

Conversely, since $\map_4$ is non-birational, $S$ is a (1,2) surface by Bombieri \cite{Bom} and the canonical curve family on $S$ is the desired family of canonical degree 1.
\end{proof}

The aim of this paper is to study those open cases on 3-folds. Here is one of the main results:

\begin{thm}\label{m1} Let $X$ be a minimal projective 3-fold  of general type and $p_g(X)\geq 2$. Then
\begin{itemize}
\item[(1)] $\map_7$ is birational if and only if either $p_g(X)>2$ or $X$ does not admit any genus 2 curve family ${\mathscr C}$ of canonical degree $\frac{2}{3}$. (see Definition \ref{CF} for exact meaning of a ``curve family'') 

\item[(2)] $\map_6$ is birational unless $p_g=2$ and $K_X^3\leq 6$.

\item[(3)] $\map_5$ is birational unless $X$ has one of the following numerical types:
\begin{itemize}
\item[3.1.] $p_g(X)=3$ and $K_X^3\leq 6$;
\item[3.2.] $p_g(X)=2$ and $K_X^3\leq 81$.
\end{itemize}
\end{itemize} \end{thm}

We provide some supporting examples which show that all above exceptional cases do really occur.
\medskip

\noindent{\bf Example A.} (1) The general hypersurface 
$X:=X_{16}\subset \bP(1,1,2,3,8)$ (see Fletcher \cite[p151]{Flt}) of degree 16 has the invariants $p_g=2$ and $K_X^3=\frac{1}{3}$. It is clear that $\map_7$ is non-birational.
Automatically $\map_6$, $\map_5$ are non-birational either since $p_g>0$. According to our earlier result in \cite{MA}, $X$ is canonically fibred by (1,2) surfaces and the relative canonical map of $\map_1$ of $X$ gives a genus 2 curve family of canonical degree $\frac{2}{3}$. 

(2) The general hypersurface $X:=X_{14}\subset \bP(1,1,2,2,7)$ of degree 14 has the invariants $p_g=2$ and $K_X^3=\frac{1}{2}$ and there are 7 orbifold points $\frac{1}{2}(1,-1,1)$ on $X$. Clearly $\map_7$ is birational. Since the Cartier index of $X$ is $2$, it does not admit any curve family of canonical degree $\frac{2}{3}$, which accounts for the birationality of $\map_7$ by virtue of Theorem \ref{m1}(1). Note that $\map_6$ of $X$ is non-birational. 

(3) The general hypersurface $X:=X_{12}\subset \bP(1,1,1,2,6)$ of degree 12 has the invariants $p_g=3$ and $K_X^3=1$, but $\map_5$ is non-birational. 

(4) The codimension 2 complete intersection $$X:=X_{4,12}\subset \bP(1,1,2,2,3,6)$$ of bi-degree $(4,12)$ has the invariants $p_g=2$ and $K_X^3=\frac{2}{3}$, but $\map_5$ is non-birational. 
\bigskip

Going on the story as in Chen-Zhang \cite{MZ}, we shall characterize the birationality of $\map_4$ as follows:

\begin{thm}\label{m4} Let $X$ be a minimal projective 3-fold of general type and $X$ satisfies one of the following conditions:
\begin{itemize}
\item[0.] $p_g(X)\geq 5$. (established in \cite{MZ}) 
\item[1.] $p_g(X)=4$ and $K_X^3>28$.
\item[2.] $p_g(X)=3$ and $K_X^3>180$.
\item[3.] $p_g(X)=2$ and $K_X^3>855$.
\end{itemize}
Then $\map_4$ is birational if and only if $X$ does not admit any genus two curve family ${\mathscr C}$ of canonical degree 1.
\end{thm}

Our classification in this paper has provided a broader way to find non-trivial examples with $\map_4$ non-birational. See, for instance, the following:
\medskip

\noindent{\bf Example B}. According to Theorem \ref{d=3}, 
any minimal 3-fold $X$ having terminal singularities, $p_g(X)=4$ and $K_X^3=2$ (such that $\map_1$ is generically finite) must have non-birational $\map_4$. The hypersurface $X_{10}\subset \bP(1,1,1,1,5)$ is a very special candidate which is smooth. 

\begin{Rema} Theorem \ref{m4} is parallel to Theorem 0. Some other examples with non-birational $\map_4$ can be found in Fletcher \cite{Flt} and Chen-Zhang \cite{MZ}.
\end{Rema}

\begin{Rema} The $\map_4$ and $\map_3$ have ever been partially studied by Zhou \cite{Zhou} and Zhu \cite{Zhu} under extra conditions.  
\end{Rema}


Here are some of the main observation of this paper:
\begin{itemize}
\item[$\diamond$] For a nef and big $\bQ$-divisor ${\mathcal L}$ on a smooth projective surface $S$ with $p_g(S)=1$, the geometric nature of the linear system $|K_S+\roundup{{\mathcal L}}|$ is difficult to detect, especially when (up to numerical equivalence) ${\mathcal L}<2\sigma^*(K_{F_0})$. Our solution is to deform it into a successful application of Masek's interesting theorem in \cite{Masek} -- a generalized form of Ein-Lazarsfeld's argument in \cite{E-L}. 

\item[$\diamond$] When $X$ is fibred by surfaces with very small invariant $c_1^2$, a large ratio ${K_X^3}/{3c_1^2}$ will be much more effective in improving our ``canonical restriction inequality'' in \cite[Lemma 3.7]{MZ}, which will amend whatever we didn't realize before, but one needs to assume $K_X^3$ to be large enough.   

\item[$\diamond$] Theorem 0 and the main statements of this paper lead us to expect that the existence of certain curve family with very small canonical degree essentially affects the birational geometry of varieties in question.

\item[$\diamond$] Parallel to the surface case, it is impossible to handle things in a uniform way to treat 3-folds with very small invariants (for instance, with small $p_g$ and $K^3$). Sometimes the very refined classification of surfaces are needed (see, for example, Claim \ref{42}.4 and Remark \ref{d3P3}). It is even inevitable to ask lots of new questions (on surfaces with $p_g\leq 1$) which, unfortunately, are still mysterious to experts on surfaces.  
\end{itemize}
\bigskip

We are in favor of the following symbols:
\medskip

{\small
\noindent ``$\sim$'' denotes linear equivalence or ${\bQ}$-linear equivalence;

\noindent ``$\equiv$'' denotes numerical equivalence;

\noindent ``$A\geq_{\text{num}}B$'' means that $A-B$ is numerically equivalent to an effective $\bQ$-divisor.}

\section{\bf Definitions, lemmas, notations and the setting}
Throughout $X$ will be a minimal projective 3-fold of general type, on which $\omega_X=\OX(K_X)$ is the canonical sheaf and $K_X$ a canonical divisor. 
\medskip

\subsection{\bf Fixed notation and setting}\label{set}\

Assume $p_g(X):=h^0(X, \omega_X)\geq 2$. We may study the geometry induced from the canonical map $\varphi_{1}:X\dashrightarrow \bP^{p_g-1}$ where $\map_1$ is usually a rational map. 

{}Fix an effective Weil divisor $K_1\sim K_X$. Take successive blow-ups $\pi: X'\rightarrow X$, which exists by Hironaka's big theorem, such that:
\smallskip

(i) $X'$ is nonsingular,

(ii) the movable part of $|K_{X'}|$ is base point free,

(iii) the support of $\pi^*(K_1)$ is of simple normal crossings.
\smallskip

Denote by $g$ the composition $\varphi_{1}\circ\pi$. So $g:
X'\rightarrow \Sigma\subseteq{\mathbb P}^{p_g(X)-1}$ is a morphism.
Let $X'\overset{f}\rightarrow \Gamma\overset{s}\rightarrow \Sigma$ be
the Stein factorization of $g$. We get the following commutative diagram:
\medskip

\begin{picture}(50,80) \put(100,0){$X$} \put(100,60){$X'$}
\put(170,0){$\Sigma$} \put(170,60){$\Gamma$} \put(112,65){\vector(1,0){53}}
\put(106,55){\vector(0,-1){41}} \put(175,55){\vector(0,-1){43}}
\put(114,58){\vector(1,-1){49}} \multiput(112,2.6)(5,0){11}{-}
\put(162,5){\vector(1,0){4}} \put(133,70){$f$} \put(180,30){$s$}
\put(92,30){$\pi$} \put(135,-8){$\varphi_{1}$}\put(136,40){$g$}
\end{picture}
\bigskip

\noindent We may write $K_{X'}=\pi^*(K_X)+E_{\pi}= M_1+Z_1,$ where $|M_1|$ is the moving part of $|K_{X'}|$, $Z_1$ the fixed part and
$E_{\pi}$ an effective ${\bQ}$-divisor which is a sum of distinct exceptional divisors with rational coeffients. For any positive integer $m$, whenever taking the round-up of
$m\pi^*(K_X)$, we always have $\roundup{m\pi^*(K_X)}\leq mK_{X'}$ by the definition of $\pi^*$. Since $h^0(X', \OO_{X'}(M_1))=h^0(\omega_X)$, we may also write $\pi^*(K_X)=M_1+E_1'$ where $E_1'=Z_1-E_{\pi}$ is an effective ${\mathbb Q}$-divisor. Set $d_1:=\dim\overline{\varphi_1(X)}$. Clearly one has $0<d_1\leq 3$. 

If $d_1=2$, a general fiber $C$ of $f$ is a smooth
projective curve of genus $\geq 2$. We say that $X$ is {\it canonically fibred by curves}.

If $d_1=1$, a general fiber $F$ of $f$ is a smooth
projective surface of general type. We say that $X$ is {\it
canonically fibred by surfaces} with invariants $(c_1^2(F_0), p_g(F)),$ where $F_0$ is the minimal model of $F$ obtained from the contraction morphism $\sigma: F\rightarrow F_0$. We may write $M\equiv p F$ where $p\ge p_g(X)-1$. Denote $b:=g(\Gamma)$. 

{\it A generic irreducible element $S$ of} $|M|$ means either a general member of $|M|$ in the case $d_1\geq 2$ or, otherwise, a general fiber $F$ of $f$.

For any integer $m>0$, $|M_m|$ denotes the moving part of $|mK_{X'}|$. 
\medskip

\subsection{\bf Technical preparation}\

We always refer to Chen-Zhang \cite[Section 3]{MZ} for birationality principles (see \cite[Lemma 3.1, Lemma 3.2]{MZ}) and the ``canonical restriction inequality'' \cite[Lemma 3.7]{MZ}. 

\begin{defn} Let $|M|$ be a movable linear system on a normal projective variety $Z$. We say that the rational map $\Phi_{|M|}$ {\it distinguishes sub-varieties $W_1, W_2\subset Z$} (where $W_i$ is not contained in the base locus of $|M|$ for $i=1,2$)  if, set theoretically, $\overline{\Phi_{|M|}(W_1)}\nsubseteqq \overline{\Phi_{|M|}(W_2)}$ and $\overline{\Phi_{|M|}(W_2)}\nsubseteqq \overline{\Phi_{|M|}(W_1)}$. We say $\Phi_{|M|}$ {\it separates points $P, Q\in Z$} if $P, Q\not\in \text{Bs}|M|$ and $\Phi_{|M|}(P)\neq \Phi_{|M|}(Q)$.
\end{defn}

For the convenience of readers, we recall the key technical theorem here which will be frequently used throughout.

We define $p$ to be 1 if
$d_1\geq 2$. We assume further that $S$ is
equipped with a movable linear system $|G|$ and that a generic
irreducible element $C$ of $|G|$ is smooth. We define
$\xi:=(\pi^*(K_X)\cdot C)$. We fix an integer $m>0$, and we
consider a linear system on $S$, $|L_m|$, given by
$L_m:=|K_{S}+\roundup{(m-1)\pi^*(K_X)-S-\frac{1}{p}E_1'}|_{S}|$.
Then one has the following theorem:

\begin{thm}\label{CZkey}(Chen--Zhang \cite[Theorem 3.6]{MZ})  (1) If $|L_m|$ separates different generic irreducible elements of $|G|$ (namely $\Phi_{L_m}(C_1)\neq
\Phi_{L_m}(C_2)$ where $C_1$, $C_2$ are different generic
irreducible elements of $|G|$) and $\beta$ is a rational number
such that $\pi^*(K_X)-\beta C$ is numerically equivalent to an
effective ${\mathbb Q}$-divisor, then $\varphi_m$ is birational
if one of the following conditions is satisfied, where we set
$\alpha:=(m-1-\frac{1}{p}-\frac{1}{\beta})\xi$ and
$\alpha_0:=\roundup{\alpha}$:

$i$. $\alpha > 2;$

$ii$. $\alpha_0\geq 2$ and $C$ is non-hyperelliptic;

$iii$. $\alpha>0$, $C$ is non-hyperelliptic and $C$ is an even
divisor on $S$.
\medskip

\noindent (2) One has the inequality $\xi\geq
\frac{2g(C)-2+\alpha_0}{m}$ if one of the following conditions
is satisfied:

$iv$. $\alpha > 1;$

$v$. $\alpha>0$ and $C$ is an even divisor on $S$.
\end{thm}

\subsection{\bf Other necessary notions and lemmas}\ 


Since all 3-folds considered here have $p_g(X)>0$, we immediately have the following fact which will be tacitly used throughout the paper:

\begin{fact} Under the setting of \ref{set}, $\map_m:=\Phi_{|mK_{X'}|}$ distinguishes different generic irreducible elements $S$ of $|M_1|$ for all $m\geq 2$. 
\end{fact}

We will frequently use the following base point freeness  due to \cite{Bom, C-C, Fr, Reider}:

\begin{fact} Let $S_0$ be a minimal projective surface of general type and $p_g(S_0)>0$. Then $|2K_{S_0}|$ is base point free.
\end{fact}

\begin{defn}\label{CF} Let $Z$ be a normal projective $\bQ$-Gorenstein variety. Let $\theta:Z'\rw Z$ be a birational morphism and $h:Z'\rw W$ be a fibration onto another normal variety $W$ with $\dim(W)=\dim(Z)-1$. Then we call ${\mathscr C}:=\{\theta(F)|F\ \text{is a fiber of}\ h\}$ a {\it curve family} on $Z$. For a general member $C\in {\mathscr C}$, the rational number $\deg({\mathscr C}):=(K_Z\cdot C)$ is referred to as {\it the canonical degree of} ${\mathscr C}$. \end{defn} 

Note that $\deg({\mathscr C})$ is independent of the birational morphism $\theta$ by the intersection theory.

\begin{defn} Let $S$ be a nonsingular projective surface. 
For a point $P\in S$, {\it $P$ is said to be very general} if $P$ lies in the complement of the union of countable curves on $S$.
\end{defn}

\begin{lem}\label{masek} Let $S$ be a nonsingular projective surface. Let ${\mathcal L}$ be a nef and big $\bQ$-divisor on $S$ satisfying the following conditions:
\begin{itemize}
\item[(1)] ${\mathcal L}^2>8$.
\item[(2)] $({\mathcal L}\cdot C_{P})\geq 4$ for all irreducible curves $C_{P}$ passing through any very general point $P\in S$.
\end{itemize}
Then the linear system $|K_S+\roundup{{\mathcal L}}|$ separates two distinct points in very general positions. Consequently it gives a birational map.
\end{lem}
\begin{proof} This is a direct result from the proof of Masek \cite[Proposition 4]{Masek}. We keep the same notation there. Let $p,q$ be two distinct very general points on $S$. Then we are in the situation $\mu_p=\mu_q=0$. Just set $\beta_{1,p}=\beta_{1,q}=2$ and $\beta_{2,p}=\beta_{2,q}=2$. Then our situation here fits into all numerical requirements there and, as a result, the proof follows. Note, however, Masek's condition of ``M being ample'' is set to secure the local positivity at every points in order to obtain base point freeness and very ampleness. To obtain birationality, the ``nef and big'' condition is sufficient. \end{proof}

\begin{lem}\label{RR} Let $\pi:X'\rw X$ be a birational morphism from a nonsingular model $X'$ onto $X$ which is a minimal projective 3-fold of general type. Assume $f:X'\rw \bP^1$ be a fibration with the general fiber $F$. Let $\sigma: F\rw F_0$ be the birational contraction onto the minimal model. Set $\tau_0:=\frac{K_X^3}{3K_{F_0}^2}$. Then
\begin{itemize}
\item[(i)] For any rational number $\delta>0$, there are 
two positive integers $N$ and $n$ such that $\frac{n}{N}=\tau_0-\delta$ and that
$$N\pi^*(K_X)\genum nF.$$
\item[(ii)] For any small rational number $\eps_0$, there exists an effective $\bQ$-divisor $J_{\eps_0}$ such that
$$\pi^*(K_X)|_F\equiv (\frac{\tau_0}{\tau_0+1}-\eps_0)\sigma^*(K_{F_0})+J_{\eps_0}.$$
\end{itemize}
\end{lem}
\begin{proof} (1) For any sufficiently large and divisible integer $m>0$ (such that $m$ is divisible by the Cartier index of $X$), the Riemann-Roch on $X'$ implies
$$P_m(X')=h^0(X', m\pi^*(K_X))\approx \frac{1}{6}K_X^3m^3.$$
On the other hand, the Riemann-Roch on $F$ gives
$$P_m(F)\approx\frac{1}{2}K_{F_0}^2m^2.$$
Therefore $P_m(X')>(\tau_0-\delta)mP_m(F)$ for $m\gg 0$. Consider the restriction maps:
$$H^0(X', M_m-tF)\overset{\theta_t}\lrw V_{m,t}\subset H^0(F, mK_F)$$ where $t\geq 0$ and $V_{m,t}$ is the image vector space. Since $\dim (V_{m,t})\leq P_m(F)$ for all 
$t$, we have 
$$m\pi^*(K_X)-(\tau_0-\delta)mF\geq M_m-(\tau_0-\delta)mF>0$$ for all large and divisible integers $m$ (such that $(\tau_0-\delta)m$ is integral). Pick a large such integer $m=l_0$ and set $N:=l_0$ while $n:=(\tau_0-\delta)l_0$. So we get (i).  

(2) Statement (ii) follows directly from our canonical restriction inequality \cite[Lemma 3.7]{MZ} since $\frac{p}{m_0}\mapsto \tau_0$ in this situation. We are done.
\end{proof}

\begin{lem}\label{rr} Let $S$ be a nonsingular projective surface on which there is a nef and big $\bQ$-divisor $L$ and a smooth curve $C$ with $(L\cdot C)>0$. Set $\nu_0:=\frac{L^2}{2(L\cdot C)}$. Then
$$L\genum (\nu_0-\delta_0)C$$ 
for all very small rational numbers $\delta_0>0$.
\end{lem}
\begin{proof} This is similar to Lemma \ref{RR}. For a very large and divisible integer $m$, we have
$$h^0(S, mL)\approx \frac{1}{2}L^2m^2$$ and 
$h^0(C, mL|_C)\approx (L\cdot C)m$. Then the statement follows by simply considering the restriction maps:
$$H^0(S, mL)\lrw H^0(C, mL|_C).$$
We omit other details. 
\end{proof}

The following result is also frequently used to distinguish curves on surfaces. 

\begin{lem}\label{_*} Let $\pi:X'\rw X$ be a birational morphism from nonsingular projective 3-fold $X'$ onto a minimal 3-fold $X$ of general type. Let $f:X'\rw\bP^1$ be a fibration. Assume $\OO(p)\hrw f_*\omega_{X'}$ for some integer $p>0$. Then, for all $t>0$ with $|tp\sigma^*(K_{F_0})|$ base point free, one has
$$t(p+2)\pi^*(K_X)|_F\geq tp\sigma^*(K_{F_0})$$
where $F$ is a general fiber of $f$ and $\sigma: F\rw F_0$ is the contraction onto the minimal model.
\end{lem}
\begin{proof} By assumption, one has
$$f_*\omega_{X'/\bP^1}^{\otimes tp}\hrw f_*\omega_{X'}^{\otimes t(p+2)}.$$
By the semi-positivity theorem for $f_*\omega_{X'/\bP^1}^{\otimes tp}$, we see that $f_*\omega_{X'/\bP^1}^{\otimes tp}$ is generated by global sections. Thus the local sections along the general fiber $F$ can be glued into some global sections of $f_*\omega_{X'}^{\otimes t(p+2)}$. This implies 
\begin{eqnarray*}
|t(p+2)K_{X'}||_F&\supset&|tp\sigma^*(K_{F_0})|_{\text{mov}}+\text{(fixed divisor)}\\
&=&|tp\sigma^*(K_{F_0})|+\text{(fixed divisor)}
\end{eqnarray*}
since $|tp\sigma^*(K_{F_0})|$ is base point free. Or, in divisor language, one has 
$$t(p+2)\pi^*(K_X)|_F\geq M_{t(p+2)}|_F\geq tp\sigma^*(K_{F_0})$$
where $|M_{t(p+2)}|$ is the moving part of $|t(p+2)K_{X'}|$. 
\end{proof}

\section{\bf Proof of Theorem \ref{m1} (Part 1)}

\subsection{\bf Characterization of the birationality of $\map_7$}\label{2/3}\ 

We start with the proof of Theorem \ref{m1}(1). Since $p_g(X)\geq 2$, we have an induced fibration $f:X'\rw \Gamma$. 

Assume $\map_7$ is non-birational. Then, by \cite[Theorem 1.2 and Section 4]{IJM}, one has $p_g(X)=2$, $\Gamma\cong \bP^1$ and a general fiber $F$ of $f$ is a (1,2) surface. We show the existence of the curve family ${\mathscr C}$ as claimed in the statement. Pick a general fiber $F$ and take $|G|$ to be the moving part of $|K_F|$. Take further necessary birational modifications to $\pi$ such that (and thus may assume) the relative canonical map of $f$ is a morphism. By the surface theory (see \cite{BPV}), a generic irreducible element $C$ of $|G|$ is a smooth curve of genus 2. We have $m_0=1$, $p=1$ and \cite[Lemma 3.7]{MZ} implies $\beta\mapsto \frac{1}{2}$. We have already known $\xi:=(\pi^*(K_X)\cdot C)\geq \frac{3}{5}$ in \cite[Section 4]{IJM}. Take $m=6$. Since $\alpha_6\geq 2\xi>1$, Theorem \ref{CZkey} implies $\xi\geq \frac{2}{3}$. Now $\xi>\frac{2}{3}$ is impossible since, otherwise, $\alpha_7\geq 3\xi>2$ and $\map_7$ will be birational by \cite[Theorem 3.6]{MZ}, a contradiction. Thus $\xi=\frac{2}{3}$. Take ${\mathscr C}:=\{\pi(C_t)|C_t\ \text{is a fibre of the relative canonical map of}\ f\}$. Clearly  $\deg({\mathscr C})=\xi=\frac{2}{3}$ by the projection formula. 

Conversely, assume $p_g(X)=2$ and there is a genus 2 curve family ${\mathscr C}$ of canonical degree $\frac{2}{3}$. We study the general fiber $F$ of the canonically induced fibration $f: X'\rw \Gamma$. 
\medskip

{\bf Claim \ref{2/3}.1.} $\Gamma\cong \bP^1$ and $F$ must be a (1,2) surface.
\begin{proof} Modulo further birational modifications, we may assume that the curve family ${\mathscr C}$ is free on $X'$. Pick a general curve $\hat{C}\subset X'$ such that $\pi(\hat{C})$ dominates a general curve in ${\mathscr C}$ on $X$. {}First, we see $\hat{C}\subset F$ for some general fiber $F$. Otherwise, $f(\hat{C})=\Gamma$, $h^0(\hat{C}, F|_{\hat{C}})\geq 2$ and then $(\pi^*(K_X)\cdot \hat{C})\geq \deg(F|_{\hat{C}})\geq 2$ which contradicts to the assumption. Secondly, if $g(\Gamma)>0$, then we have $\pi^*(K_X)|_F\cong \sigma^*(K_{F_0})$ by \cite[Lemma 4.7]{MZ} and 
$$(\pi^*(K_X)\cdot \hat{C})=(\pi^*(K_X)|_F\cdot \hat{C})
=(\sigma^*(K_{F_0})\cdot \hat{C})\geq 1$$
since $\hat{C}$ is moving in a family on $F$, which is again a contradiction. Thus we have seen $\Gamma\cong \bP^1$. 

Now we study the numerical type of the general fiber $F$. 
We have 
$$\frac{2}{3}=(\pi^*(K_X)|_F\cdot \hat{C})\geq \frac{1}{2}(\sigma^*(K_{F_0})\cdot \hat{C}),$$
which implies $(\sigma^*(K_{F_0})\cdot \hat{C})=1$. Observing that $\hat{C}$ is moving, the Hodge index theorem and the assumption $g(\hat{C})=2$ imply $K_{F_0}^2=1$.  By the surface theory and the proof of \cite[Claim 2.14]{MA}, $F$ must be a (1,2) surface and $\hat{C}$ is the moving part of $|K_F|$. So we have $\hat{C}=C$ as a general member of $|G|$ on $F$. \end{proof} 
\medskip

{\bf Claim \ref{2/3}.2.} $|M_7||_C=|2K_C|$.

\begin{proof} If $C$ is a general curve in the moving part of $|K_F|$, one has $K_F=\sigma^*(K_{F_0})+E_{(0)}$ and $E_{(0)}\cap C$ is a single point $P\in C$ with $2P\sim K_C$, which is due to the fact that $|K_{F_0}|$ has exactly one base point. In particular, we have $K_{F}|_C=2P$. This means that $(\pi^*(K_X)|_F+E_{\pi}|_F)|_C=K_F|_C=2P$, once we fix a general fiber $F$ and a general curve $C$ on $F$. 

Since $\OO_{\Gamma}(1)\hrw f_*\omega_{X'}$, Lemma \ref{_*} implies $3\pi^*(K_X)|_F\geq C$. Now the Kawamata-Viehweg vanishing theorem (\cite{KV,VV}) gives
\begin{eqnarray*}
|K_{X'}+\roundup{5\pi^*(K_X)}+F|_F&=&|K_F+\roundup{5\pi^*(K_X)}|_F|\\
&\supset&
|K_F+\roundup{2\pi^*(K_X)|_F}+C|+\text{(fixed divisor)}
\end{eqnarray*}
 and 
 $$|K_F+\roundup{2\pi^*(K_X)|_F}+C||_C=|K_C+D|$$
 with $\deg(D)\geq 2\xi\geq \frac{4}{3}$. Since $\pi^*(K_X)|_F\leq K_F$, we have $D=2P$. 
 
Noting that $|K_{X'}+\roundup{5\pi^*(K_X)}+F|+\text{(fixed divisor)}\subset |7K_{X'}|$ and $|K_C+D|$ is movable, we have
$$|M_7|_C|\supset |M_7||_C\supset |K_C+D|+\text{(fixed divisor)}.$$
Since $\deg(M_7|_C)\leq 7(\pi^*(K_X)|_F\cdot C)=\frac{14}{3}$, we get $\deg(M_7|_C)\leq 4$. Thus the only possibility is 
$M_7|_C\sim K_C+D$ and $|M_7||_C=|K_C+D|=|2K_C|$, which gives a finite map of degree 2. 
\end{proof}

Clearly $\map_7|_F$ distinguishes different general curves $C$, we see that $\map_7$ is generically finite of degree 2. So we conclude Theorem \ref{m1}(1). 
\medskip

\subsection{\bf Birationality of $\map_6$}\label{R2}\ 

 In this subsection we shall prove Theorem \ref{m1}(2). 

By \cite[Theorem 1.2, Theorem 3.3]{IJM}, we only need to assume $p_g(X)=2$ and $\Gamma\cong \bP^1$ to prove the birationality of $\map_6$. 

We have an induced fibration $f:X'\rw \Gamma$ where we pick a general fiber $F$. By the surface theory, $F$ must be among the following types, since $p_g(F)>0$:
\begin{itemize}
\item[a.] $K_{F_0}^2\geq 3$.
\item[b.] $K_{F_0}^2=2$.
\item[c.] $K_{F_0}^2=1$, $p_g(F)=1$.
\item[d.] $K_{F_0}^2=1$, $p_g(F)=2$.
\end{itemize}

Clearly it is sufficient to prove the birationality of $\map_6|_F$ for a general fiber $F$. 
\medskip

{\bf Claim \ref{R2}.1.} {\em If $F$ is of Type $(a)$, $\map_6$ is birational.}
\begin{proof} 
 By Kawamata-Viehweg vanishing, we have
\begin{equation}\label{L4}
\begin{array}{rl}
|K_{X'}+\roundup{4\pi^*(K_X)}+M_1||_F&=|K_F+\roundup{4\pi^*(K_X)}|_F|\\
 &\supset
|K_F+\roundup{L_{4}}|+(\text{fixed divisor})
\end{array}
\end{equation}
where $L_{4}:=4\pi^*(K_X)|_F$ is a nef and big $\bQ$-divisor.  

By \cite[Lemma 3.7]{MZ}, there exists an effective $\bQ$-divisor $H_{\eps}$ such that
\begin{equation}\label{s1}
\pi^*(K_X)|_F \equiv (\frac{1}{2}-\eps)\sigma^*(K_{F_0})+H_{\eps}
\end{equation}
for all very small rational numbers $\eps>0$. Noting that both $\pi^*(K_X)|_F$ and $\sigma^*(K_{F_0})$ are nef, we get
\begin{equation}\label{square}
L_{4}^2\geq 16\cdot (\frac{1}{2}-\eps)^2\cdot\sigma^*(K_{F_0})^2\geq 3(2-4\eps)^2>8
\end{equation}
whenever $\eps$ is small enough.
\medskip

{\bf ($\ddag$)} {}For a very general point $P\in F$, any curve $C_P$ passing through $P$ is of general type, i.e. $g(C_P)\geq 2$. Then an easy exercise will show $(\sigma^*(K_{F_0})\cdot C_P)\geq 2$ simply due to the fact $K_{F_0}^2>1$. Pick a sequence of rational numbers $\{\eps_n\}$ which converges to zero. By the choice of $P$, we may assume $C_P\not\subset \bigcup_{n}H_{\eps_n}$. In particular, $(C_P\cdot H_{\eps_n})\geq 0$ for all $n>0$.
Thus 
$$(L_{4}\cdot C_P)\geq 4(\frac{1}{2}-\eps_n) (\sigma^*(K_{F_0})\cdot C_P)=(2-4\eps_n)(\sigma^*(K_{F_0})\cdot C_p).$$
Taking the limit while $n\mapsto +\infty$, we have $(L_{4}\cdot C_P)\geq 4$. 
\medskip

Now Lemma \ref{masek} implies that $|K_{F}+L_{4}|$ gives a birational map. Thus $\map_6$ is birational.
\end{proof}

{\bf Claim \ref{R2}.2.} {\em If $K_X^3>6$ and $F$ is of Type (b), $\map_6$ is birational.}

\begin{proof} Take $|G|:=|2\sigma^*(K_{F_0})|$ if $p_g(F)\neq 3$ and, otherwise, take $|G|$ to be the moving part of $|K_F|$. Clearly $\map_6|_F$ distinguishes different general irreducible elements of $|G|$ by Lemma \ref{_*} and Relation (\ref{L4}) respectively.  

When $K_X^3>6$, $\tau_0:=\frac{K_X^3}{3K_{F_0}^2}>1$ and Lemma \ref{RR} implies that there is an effective $\bQ$-divisor $J_{\eps}$ with
$$\pi^*(K_X)|_F\equiv (\frac{\tau_0}{\tau_0+1}-\eps)\sigma^*(K_{F_0})+J_{\eps}$$
for any small rational numbers $\eps>0$. Take a small $\eps_0$ such that $\eta_0:=\frac{\tau_0}{\tau_0+1}-\eps_0>\frac{1}{2}$. 

If $p_g(F)\neq 3$, a generic irreducible element $C$ of $|G|$ is even and non-hyperelliptic. We see $\beta\geq \frac{1}{2}\eta_0>\frac{1}{4}$. Since $\alpha_6=(6-2-\frac{1}{\beta})\xi>0$, Theorem \ref{CZkey}(iii) implies the birationality of $\map_6$. 

If $p_g(F)=3$, $\beta\geq \eta_0>\frac{1}{2}$ and $\xi:=(\pi^*(K_X)|_F\cdot C)>\frac{1}{2}(\sigma^*(K_{F_0})\cdot C)\geq 1$. Since $\alpha_6\geq (6-2-\frac{1}{\beta})\xi>2$, $\map_6$ is birational again by Theorem \ref{CZkey}. The claim is proved.
\end{proof}

{\bf Claim \ref{R2}.3.} {\em If $K_X^3>3$ and $F$ is of Type (c), $\map_6$ is birational.}
\begin{proof} Take $|G|:=|2\sigma^*(K_{F_0})|$, which is base point free. Again Lemma \ref{_*} says $\map_6|_F$ distinguishes different general members in $|G|$. 

Since $\tau_0=\frac{K_X^3}{3K_{F_0}^2}>1$, we have $\beta>\frac{1}{4}$ similarly. Thus we still have $\alpha_6>0$ and $\map_6$ is birational by Theorem \ref{CZkey}.
\end{proof}

{\bf Claim \ref{R2}.4.} {\em If $K_X^3>5$ and $F$ is of Type (d), then $\map_6$ is birational.}
\begin{proof} Take $|G|$ to be the moving part of $|K_F|$. Then we have $\xi=(\pi^*(K_X)|_F\cdot C)\geq \frac{2}{3}$ from \ref{2/3}. 

Since $\tau_0=\frac{K_X^3}{3K_{F_0}^2}>\frac{5}{3}$, we have $\frac{p}{m_0}>\frac{5}{3}$ and $\beta>\frac{5}{8}$ by Lemma \ref{RR}.  Since $\alpha_5\geq (5-1-\frac{3}{5}-\frac{8}{5})\xi>1$, we get $\xi\geq \frac{4}{5}$ by Theorem \ref{CZkey}. Now $\alpha_6\geq \frac{14}{5}\xi>2$ which implies the birationality of $\map_6$ by Theorem \ref{CZkey} once more.
\end{proof}

We have proved Theorem \ref{m1}(2).

\section{\bf Proof of Theorem \ref{m1} (Part 2)}
In this section we shall work on the birationality of $\map_5$. By \cite[Theorem 1.2]{IJM}, we only need to study the cases $p_g(X)=2,\ 3$. 
\medskip

\subsection{\bf $\map_5$ in the case $p_g(X)=3$ and $d_1=2$.}\label{pg3d2}\ 

When $d_1=2$, a general fiber $C$ of $f$ is a curve of genus $\geq 2$. Pick a general member $S\in |M_1|$. Take $|G|:=|S|_S|$ where 
$S|_S\equiv eC$ with $e\geq 1$. Taking the restriction, we get
$$\pi^*(K_X)|_S\equiv S|_S+{E_1'}|_S.$$
Set $L:=\pi^*(K_X)|_S$ which is an effective nef and big $\bQ$-divisor. Clearly, one has $L^2\geq \xi$ by definition.
\medskip

{\bf Claim \ref{pg3d2}.1.} When $g(C)\geq 3$, $\xi>1$; when $g(C)=2$, $\xi\geq 1$. 
\begin{proof}
This is a direct consequence of Theorem \ref{CZkey}. In fact, we have $p=1$ and $\beta=1$. When $g(C)\geq 3$, Theorem \ref{CZkey} implies $\xi\geq \frac{4}{3}>1$. When $g(C)=2$, clearly one has $\xi\geq \frac{2}{3}$. Repeatedly taking suitable integer $m$ and running Theorem \ref{CZkey}(2), one has no difficulty to obtain $\xi\geq 1$ (an easy exercise!).
\end{proof}

{\bf Claim \ref{pg3d2}.2.} $K_X^3>1$ if and only if $L^2>1$.
\begin{proof}
To see this, we have the following inequality:
$$K_X^3=\pi^*(K_X)^3=(\pi^*(K_X)^2\cdot S)+(\pi^*(K_X)^2\cdot E_1')\geq L^2.$$
On the other hand, if we take a sufficiently large integer $m$ such that $|m\pi^*(K_X)|$ is base point free, then a general member $T$ is a smooth surface and we apply this to estimate $L^2$ by the Hodge index theorem:
\begin{eqnarray*}
L^2&=&\frac{1}{m}(\pi^*(K_X)|_T\cdot S|_T)\geq \frac{1}{m}
\sqrt{(\pi^*(K_X)|_T)^2\cdot (S|_T)^2}\\
&\geq&\sqrt{K_X^3\cdot \xi}\geq \sqrt{K_X^3}.\end{eqnarray*}
So the claim is true.
\end{proof}

{\bf Claim \ref{pg3d2}.3.} When $L^2>1$, $|K_S+\roundup{3L}|$ gives a birational map.
\begin{proof}
In fact, we have $L=\pi^*(K_X)=C+E_1'|_S$. There are two cases: (i) $\xi>1$; (ii) $\xi=1$.

The first case is easier since 
the vanishing theorem gives
$$|K_S+\roundup{2L}+C||_C=|K_C+\roundup{2L}|_C|$$
and clearly the linear system on the right hand side gives an embedding as 
$\deg(\roundup{2L}|_C)\geq 2\xi>2$. Thus $|K_S+3L|$ gives a birational map. 

Assume, from now on, $\xi=1$. Then
$L^2>1$ implies 
$(L\cdot E_1'|_S)>0$. Let us consider the Zariski decomposition of the effective $\bQ$-divisor $\hat{J}:=L+E_1'|_S$. Since $L$ is nef and $E_1'|_S$ is effective, we may write $E_1'|_S=N^{+}+N^{-}$ such that:
\begin{itemize}
\item[(1)] $L+N^{+}$ is nef;
\item[(2)] $(L+N^{+})\cdot N^{-}=0$;
\item[(3)] both $N^{+}$ and $N^{-}$ are effective 
$\bQ$-divisors.
\end{itemize}
Clearly, $(L\cdot N^{+})>0$. We want to show 
$(C\cdot N^{+})>0$. In fact, if $N^{-}\equiv 0$, then we have $(C\cdot N^{+})=\xi>0$. Thus we may always assume $N^{-}\not\equiv 0$, which means ${N^{-}}^2<0$ since $L+N^{+}$ is nef and big. {}From the property (2), we see $(N^+\cdot N^-)=-(L\cdot N^{-})\leq 0$ since $L$ is nef. Since $L=C+N^++N^-$, so $(L\cdot N^+)>0$ implies $(C\cdot N^+)+{N^+}^2>0$. Suppose $(C\cdot N^+)=0$, then $N^+$ is vertical with respect to the fibration on $S$ and ${N^+}^2\leq 0$, which contradicts to ${N^+}^2=(C\cdot N^+)+{N^+}^2>0$. So we have proved $(C\cdot N^+)>0$. 

Noting that 
$$|K_S+\roundup{2L+N^+}+C|+\text{(fixed divisor)}\subset |K_S+\roundup{3L}|$$ and that the vanishing theorem gives
$$|K_S+\roundup{2L+N^+}+C||_C=|K_C+D_2|$$
with $\deg(D_2)\geq 2\xi+(C\cdot N^+)>2$. Thus $|K_S+\roundup{2L+N^+}+C|$ gives a birational map and so does $|K_S+\roundup{3L}|$. We are done.
\end{proof}

On the other hand, Kawamata-Viehweg vanishing gives 
\begin{eqnarray*}
|K_{X'}+\roundup{3\pi^*(K_X)}+S||_S&=&|K_S+\roundup{3\pi^*(K_X)}|_S|\\
&\supset& |K_S+\roundup{3L}|+\text{(fixed divisor)}
\end{eqnarray*}
So we see that $\varphi_5$ is birational whenever $K_X^3>1$. 

When $K_X^3=1$, automatically $L^2=1$ and $\xi=1$ by Claim \ref{pg3d2}.2. Thus $g(C)=2$.

So we have proved the following:

\begin{thm}\label{pg3-1} Let $X$ be a minimal projective 3-fold of general type. Assume $K_X^3>1$, $p_g(X)=3$ and $d_1=2$. Then $\map_5$ is birational.
\end{thm}

\begin{rem} Theorem \ref{pg3-1} is sharp due to Example A (3). It is interesting to know whether $K_X^3>1$ is necessary to get the birationality of $\map_5$. 
\end{rem}

\subsection{\bf $\map_5$ in the case $p_g(X)=3$ and $d_1=1$.}\label{pg3d1}\ \ 

Pick a general fiber $F$ of the induced fibration $f:X'\rw \Gamma$. By \cite[Theorem 3.3]{IJM}, it is sufficient to assume $b=g(\Gamma)=0$, i.e. $\Gamma\cong \bP^1$. Note that $p_g(X)>0$ implies $p_g(F)>0$ and thus $F$ must be among the following types by the surface theory: 
\begin{itemize}
\item[(i)] $(K_{F_0}^2, p_g(F))=(1,2)$; 
\item[(ii)] $(K_{F_0}^2, p_g(F))=(2,3)$; 
\item[(iii)] other surfaces with $p_g(F)>0$. 
\end{itemize}

It suffices to show ${\varphi_5}|_F$ is birational for a general fiber $F$. One has $m_0=1$ and $p=2$. By \cite[Lemma 3.7]{MZ}, for any rational number $\varepsilon>0$, there is an effective $\bQ$-divisor $H_{\varepsilon}$ such that
\begin{equation}\label{eq4.1} \pi^*(K_X)|_F\equiv (\frac{2}{3}-\eps)\sigma^*(K_{F_0})+\He. \end{equation}
\medskip

{\bf Claim \ref{pg3d1}.1.}\label{ii} {\em If $F$ is of Type (ii), $\map_5$ is birational.}
\begin{proof}  Take $|G|$ to be the moving part of $|K_F|$. Then a general member $C\in |G|$ is a smooth curve of genus $3$ and $|G|$ gives a generically finite map (see \cite[p226]{BPV}). Besides, Relation (\ref{eq4.1}) implies $\beta\geq \frac{2}{3}-\varepsilon$ for any small $\varepsilon>0$ and thus $\xi\geq \frac{2}{3}(\sigma^*(K_{F_0})\cdot C)\geq \frac{4}{3}$.
Take $m=4$ and then $\alpha_4:=(4-1-\frac{1}{p}-\frac{1}{\beta})\xi>1$. By Theorem \ref{CZkey}, one gets $\xi\geq \frac{3}{2}$. Since Kawamata-Viehweg vanishing gives
\begin{equation}\label{eq4.2}|K_{X'}+\roundup{3\pi^*(K_X)}+F||_F\supset |K_F+\roundup{3\pi^*(K_X)|_F}|+\text{(fixed divisor)}\end{equation}
where the last linear system distinguishes different general curves $C$ and $\alpha_5\geq 2\xi\geq 3$, Theorem \ref{CZkey} implies the birationality of $\varphi_5$. \end{proof}

{\bf Claim \ref{pg3d1}.2.}\label{iii} {\em If $F$ is of Type (iii), $\map_5$ is birational.} 
\begin{proof}
We take $|G|=|2\sigma^*(K_{F_0})|$. By the surface theory, we know that $|G|$ is base point free and a general member $C$ of $|G|$ is non-hyperelliptic. Lemma \ref{_*} implies that ${\varphi_4}|_F$ distinguishes different general curves $C$. On the other hand, Equation (\ref{eq4.1}) implies $\beta\mapsto \frac{1}{3}$. So $\alpha_5=(5-1-\frac{1}{2}-\frac{1}{\beta})\xi>0$. Noting that $C$ is even and non-hyperelliptic, Theorem \ref{CZkey} implies the birationality of $\varphi_5$. \end{proof}

{\bf Claim \ref{pg3d1}.3.}\label{i} {\em If $F$ is of Type (i) and $K_X^3>6$, then $\map_5$ is birational.}
\begin{proof} On $F$, take $|G|$ to be the moving part of $|K_F|$. Since $K_X^3>6$, we have $\tau_0>2$ and Lemma \ref{RR} implies $\beta>\frac{2}{3}$. Similar to the case with $F$ being of type (ii), it suffices to study $\map_5|_C$. Recall we have $m_0=1$ and $p=2$. We have $\xi=(\pi^*(K_X)|_F\cdot C)\geq \beta\cdot(\sigma^*(K_{F_0})\cdot C)>\frac{2}{3}$. Repeatedly taking a suitable integer $m$ and running Theorem \ref{CZkey}(2), one would get $\xi\geq 1$. Now take $m=5$. We see $\alpha_5>2\xi\geq 2$. Thus, by Theorem \ref{CZkey}, $\varphi_5$ is birational.
\end{proof}

So we can conclude the following:

\begin{thm}\label{pg3-2} Let $X$ be a minimal projective 3-fold of general type. Assume $K_X^3>6$, $p_g(X)=3$ and $d_1=1$. Then $\map_5$ is birational.
\end{thm}
\medskip

\subsection{\bf $\map_5$ in the case $p_g(X)=2$.}\label{215}\ 

Automatically $d_1=1$ and this is parallel to \ref{R2}, but we are studying $\map_5$ instead. We keep the notation there and will omit redundant arguments. 
\medskip

{\bf Claim \ref{215}.1.} {\em If $K_{F_0}^2\geq 19$, $\map_5$ is birational.}
\begin{proof} The Bogomolov-Miyaoka-Yau inequaity $K^2\leq 9\chi$ implies $$p_g(F)=\chi(\OO_F)-1+q(F)\geq \frac{1}{9}K_{F_0}^2-1>1.$$
Take $|G|$ to be the moving part of $|K_F|$. Modulo birational modifications, we may assume that $|G|$ is base point free and so a generic irreducible element $C$ of $|G|$ is smooth. We still have Relation (\ref{eq4.2}) and set $L=\pi^*(K_X)|_F$. The linear system $|K_F+\roundup{3L}|$ clearly distinguishes different general curves $C$. 


If $|G|$ is not composed of a pencil, we have $C^2\geq  2$. By (\ref{s1}), we have $\beta\mapsto \frac{1}{2}$. Since $\alpha_5\geq \xi\geq \frac{1}{2}(\sigma^*(K_{F_0})\cdot C)\geq \frac{1}{2}\sqrt{K_{F_0}^2\cdot C^2}>2$,  
$\map_5$ is birational by Theorem \ref{CZkey}. 

If $|G|$ is composed of a pencil and $K_{F_0}^2\geq 19$, 
pick a generic irreducible element $C$. We study the rational number $\xi=(L\cdot C)$. If $\xi>2$, then $\alpha_5\geq \xi>2$ and $\map_5$ is birational. Otherwise, we have $\xi\leq 2$. Since $\nu_0=\frac{L^2}{2\xi}\geq \frac{19}{16}$, Lemma \ref{rr} implies 
$$L\genum (\frac{19}{16}-\delta)C$$
for any small rational number $\delta>0$, which means $\beta\geq \frac{19}{16}-\delta$. Note that $\xi\geq \frac{1}{2}(\sigma^*(K_{F_0})\cdot C)\geq 1$. 
Now $\alpha_5\geq (5-1-1-\frac{1}{\beta})\xi>2$ and so $\map_5$ is birational by Theorem \ref{CZkey} once more.
\end{proof}

{\bf Claim \ref{215}.2.} {\em If $K_X^3>81$ and $K_{F_0}^2\leq 18$, $\map_5$ is birational.}
\begin{proof} We prove the statement by analyzing  different numerical types of $F$. 
\medskip

{\bf a}. Assume $3\leq K_{F_0}^2\leq 18$. Then since $\tau_0=\frac{K_X^3}{3K_{F_0}^2}> \frac{3}{2}$, we have
 $$\pi^*(K_X)\equiv \frac{3}{2}F+E_{3/2}$$
 where $E_{3/2}$ is an effective $\bQ$-divisor and 
 $$\pi^*(K_X)|_F\genum \frac{3}{5}\sigma^*(K_{F_0})$$ 
 by Lemma \ref{RR}. By the vanishing theorem, we have
\begin{eqnarray*}
|K_{X'}+\roundup{4\pi^*(K_X)-\frac{2}{3}E_{3/2}}||_F
&=&|K_F+\roundup{4\pi^*(K_X)-F-\frac{2}{3}E_{3/2}}|_F|\\
&\supset&
|K_F+\roundup{Q}|+\text{(fixed divisor)}
\end{eqnarray*}
where 
$$Q:=(4\pi^*(K_X)-F-\frac{2}{3}E_{3/2})|_F\equiv \frac{10}{3}\pi^*(K_X)|_F\genum 2\sigma^*(K_{F_0}).$$
Now by a similar argument to ($\ddag$) in Claim \ref{R2}.1, $|K_F+\roundup{Q}|$ satisfies the condition of Lemma \ref{masek}. Thus $|K_F+\roundup{Q}|$ gives a birational map and so does $\map_5$. 
\medskip

{\bf b}. Assume $K_{F_0}^2=2$. As long as $K_X^3>24$, we have $\tau_0>4$ which means, by Lemma \ref{RR}, 
$$\pi^*(K_X)\genum (4+\delta)F$$
for some very small rational number $\delta>0$ and
$$\pi^*(K_X)|_F\genum (\frac{4}{5}+\eps_0)\sigma^*(K_{F_0})$$
for some small rational number $\eps_0>0$. The vanishing theorem gives
$$|K_{X'}+\roundup{4\pi^*(K_X)}|_F\supset
|K_F+3\sigma^*(K_{F_0})+\roundup{Q_b}|+\text{(fixed divisor)}$$
where $Q_b$ is certain nef and big $\bQ$-divisor. Now an easy exercise by applying the vanishing theorem on surfaces shows that $K_F+\sigma^*(K_{F_0})+\roundup{Q_b}$ is effective due to the fact that $\sigma^*(K_{F_0})$ is 1-connected. This means that $\map_5|_F$ distinguishes different general curves $C$ in $|G|:=|2\sigma^*(K_{F_0})|$. {}Finally one sees $\map_5$ is birational by Theorem \ref{CZkey}.
\medskip

c. Assume $K_{F_0}^2=1$.  As long as $K_X^3>12$, we have $\tau_0>4$ and then, similarly, we are reduced to prove that $|K_F+3\sigma^*(K_{F_0})+\roundup{Q_c}|$ gives a birational map where $Q_c$ is certain nef and big $\bQ$-divisor. This is the case, according to Theorem \ref{CZkey}, by taking $|G|:=|2\sigma^*(K_{F_0})|$ if $p_g(F)=1$ and $|G|$ to be the moving part of $|K_F|$ if $p_g(F)=2$. 
We also leave this as an easy exercise.

In a word, $\map_5$ is birational when $K_{X}^3>81$ and $p_g(X)=2$.
\end{proof}

We have proved the following:
\begin{thm}\label{25} Let $X$ be a minimal projective 3-fold of general type. Assume $p_g(X)=2$ and $K_X^3>81$. Then $\map_5$ is birational.
\end{thm}

Theorem \ref{pg3-1}, Theorem \ref{pg3-2} and Theorem \ref{25} imply Theorem \ref{m1}(3). 

\section{\bf Characterizing the birationality of $\map_4$ (Part A)}

By Chen-Zhang \cite[Theorem 1.3]{MZ}, we only need to study the cases $p_g(X)=2,\ 3,\ 4$. 
\medskip

\subsection{\bf $\map_4$ in the case $p_g(X)=4$ and $d_1=3$}\label{44}\ 

Keep the same setting and notation as in \ref{set}. Pick a general member $S\in |M_1|$. Consider the linear system
$|4K_{X'}|$ and its sub-system $|K_{X'}+\roundup{2\pi^*(K_X)}+M_1|$. Kawamata-Viehweg vanishing gives the relation
\begin{equation}\label{e1}\begin{array}{rl}
|K_{X'}+\roundup{2\pi^*(K_X)}+M_1||_S&=
|K_S+\roundup{2\pi^*(K_X)}|_S|\\
&\supset |K_S+\roundup{2L}|+\text{(fixed divisor)}\end{array}
\end{equation}
where $L:=\pi^*(K_X)|_S$ is an effective nef and big $\bQ$-divisor on $S$. 
Set $|G|=|M_1|_S|$. Pick a generic irreducible element $C$ of $|G|$. Then, since $p_g(S)>0$, $|K_S+\roundup{2L}|$ distinguishes different general curves $C$. So it suffices to prove the birationality (or non-birationality) of $\varphi_{4}|_C$. 
In fact, Kawamata-Viehweg vanishing gives, furthermore,
$$|K_F+\roundup{2L-E_1'|_F}||_C=|K_C+D_3|$$
where $D_3:=\roundup{2L-E_1'|_F-C}|_C$ with $\deg(D_3)\geq \xi:=(L\cdot C)$. 
\medskip

{\bf Claim \ref{44}.1.}\label{1-1} {\em $K_X^3>2$ if and only if $\xi>2$.}

\begin{proof} Pick a general surface $S\in |M_1|$. We have
$$\pi^*(K_X)|_S\sim S|_S+E_1'|_S$$ and so 
$$K_X^3=(\pi^*(K_X))^3\geq (\pi^*(K_X)^2\cdot S)=\xi.$$
On $S$, since $|C|$ is not composed of a pencil of curves, $C^2\geq 2$. Thus $\xi=(\pi^*(K_X)\cdot S^2)\geq C^2\geq 2$. 

On the other hand, by choosing a sufficiently large and divisible integer $n$ to make $|n\pi^*(K_X)|$ base point free, one can apply the Hodge index theorm on the smooth surface $S_{[n]}\in |n\pi^*(K_X)|$ to get the inequality:
$$\xi=(\pi^*(K_X)\cdot S^2)=\frac{1}{n}(\pi^*(K_X)|_{S_{[n]}}\cdot S|_{S_{[n]}})\geq \sqrt{K_X^3\cdot \xi}.$$
By \cite[Theorem 1.5(2)]{MA}, we have  $K_X^3\geq 2$. Thus it follows that 
$\xi=2$ if and only $K_X^3=2$. In fact, we have proved $K_X^3=\xi$. The lemma is proved.
\end{proof}
\medskip

{\bf Claim \ref{44}.2.}\label{1-2} {\em $\varphi_4$ is generically finite of degree $\leq 2$.}

\begin{proof} By definition, $p=1$ and $\beta=1$. Then $\alpha_4=\xi$. Whenever $\xi>2$, \cite[Theorm 3.6]{MZ} implies the birationality of $\varphi_4$. Otherwise, $\xi=2$ implies $K_X^3=2$ and since we have
\begin{equation}\label{e2}
K_X^3\geq S^3\geq \deg(\varphi_1)\geq 2,
\end{equation}
$\deg(\map_1)=2$, i.e. $\varphi_1$ must be generically finite of degree $2$.
\end{proof}
\medskip

{\bf Claim \ref{44}.3.}\label{1-3} {\em When $K_X^3=2$, $\varphi_4$ is generically finite of degree $2$.}

\begin{proof} As we have seen in the previous Claim, $\varphi_1$ is generically finite of degree 2. This means $\varphi_1|_C$ is a double cover onto $\bP^1$. In particular, $C$ is hyperelliptic and $M_1|_C$ is exactly a $g_2^1$ of $C$. Note that $C$ is a curve of genus $\geq 4$ since $(K_S\cdot C)+C^2\geq 6$. We have
\begin{eqnarray*}
|K_F+2L||_C&\supset &|K_F+2S|_S||_C+\text{(fixed divisor)}\\
&=&|K_C+S|_C|+\text{(fixed divisor)}
\end{eqnarray*}
by the vanishing theorem. This, together with the relation (\ref{e1}), implies $|M_4||_C\supset |K_C+S|_C|+\text{(fixed divisor)}$, where the last one is base point free with $\deg(K_C+S|_C)\geq 8$. Since $(4\pi^*(K_X)\cdot C)=8$, we see $|M_4||_C=|K_C+S|_C|$, which gives a double cover. Thus $\varphi_4$ is generically a double cover.
\end{proof}

Claim \ref{44}.1, Claim \ref{44}.2 and Claim \ref{44}.3 directly imply the following:

\begin{thm}\label{d=3} Let $X$ be a minimal projective 3-fold of general type. Assume $p_g(X)=4$ and $\varphi_1$ is generically finite. Then $\varphi_4$ is not birational if and only if $K_X^3=2$.
\end{thm}
\medskip

\subsection{\bf $\map_4$ in the case $p_g(X)=4$ and $d_1=2$.}\label{42}\

In this case, $\dim (\Gamma)=\dim(X)-1=2$. Pick a general fiber $C$ of $f$. We have
$$\pi^*(K_X)|_S\equiv w_2C+E_1'|_S$$
where 
\begin{equation}\label{e4.1}w_2:=\deg(s)\deg(\varphi_1(X'))\geq p_g(X)-2=2.\end{equation}
Similar to the argument in the last section, we only need to study the property of $\varphi_4|_C$ for a general curve $C$. 
\medskip

{\bf Claim \ref{42}.1.}\label{g=3} {\em If $g(C)\geq 3$, $\map_4$ is birational.} 
\begin{proof} We take $|G|=|S|_S|$ on $S$. Then $\beta=w_2\geq 2$. It follows, from Theorem \ref{CZkey}, that 
$\xi\geq \frac{\deg(K_C)}{1+\frac{1}{p}+\frac{1}{\beta}}\geq\frac{8}{5}$. Then $\alpha_4=(4-1-1-\frac{1}{2})\xi\geq \frac{12}{5}>2$. By Theorem \ref{CZkey}, $\varphi_4$ is birational.
\end{proof}

In the case $g(C)=2$, one gets
$\xi\geq \frac{2}{1+1+\frac{1}{2}}=\frac{4}{5}$. Then $\alpha_4=(4-1-1-\frac{1}{2})\xi\geq \frac{6}{5}>1$. By Theorem \ref{CZkey} once more, one gets $\xi\geq 1$. 
\medskip

{\bf Claim \ref{42}.2.} {\em 
Assume $g(C)=2$ and $\xi>1$. Then $\xi$ has the following explicit lower bounds:
\begin{itemize}
\item[(A)] If $w_2\geq 3$, $\xi\geq \frac{6}{5}$; $\xi=\frac{6}{5}$ implies $w_2=3$; $K_X^3>\frac{72}{5}$ implies $\xi>\frac{6}{5}$.
\item[(B)] If $w_2=2$, $\xi\geq \frac{8}{7}$. 
\end{itemize}}
\begin{proof} We only prove (A) while omitting parallel argument for (B). 

{}Find an integer $l_0>5$ such that $\xi\geq \frac{l_0+1}{l_0}$. Set $m'=l_0-1$ and then we have
$$\alpha_{m'}=(l_0-1-2-\frac{1}{\beta})\xi\geq (l_0-\frac{10}{3})\cdot\frac{l_0+1}{l_0}>l_0-3>1.$$
By Theorem \ref{CZkey}, one gets $\xi\geq \frac{l_0}{l_0-1}$. Recursively running this program as long as $m'\geq 5$, so we eventually get $\xi\geq \frac{6}{5}$. Clearly, if $w_2>3$, one gets $\xi>\frac{6}{5}$. 

If $K_X^3>\frac{72}{5}$, since $(\pi^*(K_X)\cdot S^2)\geq 3\xi=\frac{18}{5}$, we have
$$(\pi^*(K_X)|_S)^2\geq \sqrt{K_X^3\cdot (\pi^*(K_X)\cdot S^2)}\geq \sqrt{\frac{18}{5}K_X^3}$$
by the Hodge index theorem on a general member of $|n\pi^*(K_X)|$.
Suppose $\xi=\frac{6}{5}$. Since $\nu_0=\frac{(\pi^*(K_X)|_S)^2}{2\xi}>3$, Lemma \ref{rr}
implies
$$\pi^*(K_X)|_F\genum (3+\eta)C$$ for a small rational number $\eta>0$. Now with $\beta>3$ and applying 
Theorem \ref{CZkey} one more time, one would get $\xi>\frac{6}{5}$, a contradiction. Thus anyway we have $\xi>\frac{6}{5}$. We are done.
\end{proof}

{\bf Claim \ref{42}.3.} {\em Assume $g(C)=2$. Then $\varphi_4$ is birational in any of the following cases:
\begin{itemize}
\item[(a)] $w_2\geq 3$ and $\xi>\frac{6}{5}$;
\item[(b)] $w_2=2$ and $\xi>\frac{4}{3}$.
\end{itemize}}
\begin{proof} For case (a), since $\alpha_4=(4-1-1-\frac{1}{\beta})\xi\geq \frac{5}{3}\xi>2$,  Theorem \ref{CZkey} implies that $\varphi_4$ is birational.

For case (b), since $\alpha_4\geq \frac{3}{2}\xi>2$, $\varphi_4$ is birational.
\end{proof}

{\bf Claim \ref{42}.4.} {\em When $g(C)=2$, $w_2=2$, $K_X^3>28$ and $\frac{8}{7}\leq\xi\leq \frac{4}{3}$, $\varphi_4$ is birational.}
\begin{proof} 
Since $(\pi^*(K_X)\cdot S^2)\geq w_2\xi=\frac{16}{7}$, we have
$$(\pi^*(K_X)|_S)^2\geq \sqrt{K_X^3\cdot (\pi^*(K_X)\cdot S^2)}\geq \sqrt{\frac{16}{7}K_X^3}$$
by the Hodge index theorem on a general member of $|n\pi^*(K_X)|$.

When $K_X^3>28$, we have
$$\nu_0=\frac{(\pi^*(K_X)|_S)^2}{2\xi}\geq \frac{3}{2}\sqrt{\frac{1}{7}K_X^3}>3.$$
It follows from Lemma \ref{rr} that
$$\pi^*(K_X)|_F\genum (3+\eta)C$$ for a small rational number $\eta>0$. 

Now we are actually in the situation $w_2>3$. The statement follows directly frm Claim \ref{42}.2. and Claim \ref{42}.3. 
\end{proof}

\begin{lem}\label{non} When $p_g(X)\geq 3$, $d_1=2$ and $\xi=1$, $\varphi_4$ is non-birational.
\end{lem}
\begin{proof} This is simply a copy of \cite[Proposition 4.6]{MZ} where the proof trivially follows  with $p_g(X)=3,\ 4$. So we omit the details. \end{proof}

So we can conclude the following:

\begin{thm}\label{d=2} Let $X$ be a minimal projective 3-fold of general type. Assume $p_g(X)=4$, $K_X^3>28$ and $d_1=2$. Then $\varphi_4$ is non-birational if and only if $(K_X\cdot C)=1$ for the general fiber $C$ of $f$. In this situation, $g(C)=2$.
\end{thm}

\begin{rem}\label{d3P3} We put on the assumption $K_X^3>\frac{72}{5}$ in Theorem \ref{d=2} just to eliminate the situation: $g(C)=2$, $w_2=3$ and $\xi=\frac{6}{5}$. We do not know if this can really happen. It might be possible to analyze it in a similar way to that of Claim \ref{42}.4. But then, since the surface $\Sigma$ is of degree 3 in $\bP^3$,  there would be fairly many cases to do. Such a surface $\Sigma$ can be either normal or non-normal (see, for instance, Abe-Furushima \cite{AF}, Miyanishi-Zhang \cite{MZ1, MZ2}, Reid \cite{RN} and Ye \cite{Ye}).
\end{rem}
\medskip

\subsection{\bf $\map_4$ in the case $p_g(X)=4$ and $d_1=1$}\label{41}\
 
We have an induced fibration $f:X'\rw \Gamma$ with the general fiber $F$. Since $p_g(F)>0$ and by the surface theory, $F$ must be among the following types:
\begin{itemize}
\item[(a)] $p_g(F)=2$ and $K_{F_0}^2=1$.
\item[(b)] $K_{F_0}^2\geq 3$.
\item[(c)] other surfaces with $p_g(F)>0$.
\end{itemize}
\medskip

{\bf Claim \ref{41}.1.}\label{(1,2)}
{\em If $F$ is of Type (a), $\varphi_4$ is automatically non-birational and $X$ has a natural curve family ${\mathscr C}$ with $(K_X\cdot C_0)=1$ for the general member $C_0\in {\mathscr C}$. }
\begin{proof} Just take the relative canonical map of $f$ as what we have seen in \cite[Theorem 1.4]{MZ}. The curve family ${\mathscr C}$ is composed of all those fibers of the relative canonical map of $f$. We omit more details to avoid unnecessary redundancy. 
\end{proof}

{\bf Claim \ref{41}.2.}\label{b} {\em If $F$ is of Type (b), then $\map_4$ is birational.}

\begin{proof} According to \cite[4.8]{MZ}, it is sufficient to assume $g(\Gamma)=0$, i.e. $\Gamma\cong \bP^1$. Pick a general fiber $F$ of $f$. By Kawamata-Viehweg vanishing, we have
\begin{equation}\label{eb}
|K_{X'}+\roundup{3\pi^*(K_X)-\frac{1}{3}E_1'}||_F\supset
|K_F+\roundup{L_{1/3}}|+\text{(fixed divisor)}
\end{equation}
where $L_{1/3}:=(3\pi^*(K_X)-F-\frac{1}{3}E_1')|_F\equiv \frac{8}{3}\pi^*(K_X)|_F$. 

By \cite[Lemma 3.7]{MZ}, we have 
$$L_{1/3}\genum (2-\eps)\sigma^*(K_{F_0})$$
for any small rational number $\eps>0$. Since $K_{F_0}^2\geq 3$, a similar argument to ($\ddag$) in Claim \ref{R2}.1 shows that $|K_{F}+\roundup{L_{1/3}}|$ satisfies the conditions of Lemma \ref{masek}. Thus it follows that $\map_4$ is birational.
\end{proof}

{\bf Claim \ref{41}.3.}\label{(1,1)} {\em If $F$ is of Type (c) and $K_X^3>12$, $\map_4$ is birational.}
\begin{proof} Take $|G|:=|2\sigma^*(K_{F_0})|$. Lemma \ref{_*} says that $\map_4$ distinguishes different general members of $|G|$. Since, under the condition of the claim, one has $\tau_0>3$, Lemma \ref{RR} implies
$$\pi^*(K_X)\genum (3+\delta_0)F$$
for some very small rational number $\delta_0>0$ and 
$$L_{1/3}\equiv \frac{8}{3}\pi^*(K_X)|_F\genum (2+\eta_0)\sigma^*(K_{F_0})$$
which directly implies the birationality of $\map_4|_F$ by Theorem \ref{CZkey}. We are done.
\end{proof}

So we have proved the following:

\begin{thm}\label{411} Let $X$ be a minimal projective 3-fold of general type. Assume $p_g(X)=4$, $K_X^3>12$ and $d_1=1$. Then $\map_4$ is birational if and only if $X$ is not canonically fibred by (1,2) surfaces.
\end{thm}

Now we prove the following:

\begin{thm}\label{4b}(=Theorem \ref{m4}(1)) Let $X$ be a minimal projective 3-fold of general type. Assume $p_g(X)=4$ and $K_X^3>28$. Then $\map_4$ is birational if and only if $X$ does not contain any genus 2 curve family of canonical degree 1.
\end{thm}
\begin{proof}
If $\map_4$ is non-birational, then Theorem \ref{d=3}, Theorem \ref{d=2} and Theorem \ref{411} imply that $X$ is either canonically fibred by curves of canonical degree 1 or canonically fibred by (1,2) surfaces. For the later case, there is also a genus 2 curve family of canonical degree 1 by Claim \ref{41}.1. 

Conversely, if $X$ admits a genus 2 curve family ${\mathscr C}$ of canonical degree 1, we see $d_1\leq 2$. Otherwise, $(K_X\cdot C_0)=(\pi^*(K_X)\cdot \tilde{C}_0)\geq (S\cdot \tilde{C}_0)\geq 2$ for a general member $C_0\in {\mathscr C}$ since $|S||_{\tilde{C}_0}$ gives a generically finite map, where $\tilde{C}_0$ denotes the strict transform of $C_0$.

Now assume $d_1=2$. We want to show that ${\mathscr C}$ is the canonical curve family, i.e. $C=\tilde{C}_0$. Otherwise, $(\pi^*(K_X)\cdot \tilde{C}_0)\geq (S\cdot \tilde{C}_0)\geq 2$ since $S|_{\tilde{C}_0}$ is moving and $g(\tilde{C}_0)=2$, a contradiction. Now since $(\pi^*(K_X)\cdot C)=1$, $\map_4$ is non-birational by Theorem \ref{d=2}.

Assume $d_1=1$. If $\tilde{C}_0\not\subset F$ for a general fiber $F$, then $F|_{\tilde{C}_0}$ is moving and thus $(\pi^*(K_X)\cdot \tilde{C}_0)\geq (F\cdot \tilde{C}_0)\geq 2$, a contadiction. Thus $\tilde{C}_0\subset F$. If $F$ is not a (1,2) surface, we have
$$(\pi^*(K_X)\cdot \tilde{C}_0)\geq \frac{2}{3}(\sigma^*(K_{F_0})\cdot \tilde{C}_0)\geq \frac{4}{3}$$
since $\tilde{C}_0$ is moving on $F$, which is again impossible. So $F$ is a (1,2) surface. Clearly $\map_4$ is non-birational.
\end{proof}

\begin{rem} When $\map_4$ is non-birational, the curve family in Theorem \ref{m4}(1) is uniquely determined by $X$. Such family is called ``canonical curve family'' of $X$.
\end{rem}

\section{\bf Characterizing the birationality of $\map_4$ (Part B)}

We study $\map_4$ for the cases $p_g(X)=2,\ 3$ in this section. 
\medskip

\subsection{\bf $\map_4$ in the case $p_g(X)=3$ and $d_1=2$}\label{3-2}
\ 

This is corresponding to Subsection \ref{pg3d2}. We keep the same notation there. 
\medskip

{\bf Claim \ref{3-2}.1.}\label{bir} {\em If $K_X^3>48$ and $g(C)\geq 3$, $\map_4$ is birational.}
\begin{proof} We have seen from Claim \ref{pg3d2}.1 that $\xi\geq \frac{4}{3}$. 

If $\xi>2$, then $|K_F+\roundup{2L}||_C\supset |K_C+\roundup{L}|_C|+\text{(fixed divisor)}$ and the last linear system gives a birational map. Thus $\map_4$ is birational.

Assume $\xi\leq 2$. Then since 
$$L^2=(\pi^*(K_X)^2\cdot S)\geq \sqrt{K_X^3\cdot (\pi^*(K_X)\cdot S)^2}\geq \sqrt{\frac{4}{3}K_X^3},$$
we have $\nu_0>2$ and Lemma \ref{rr} implies
$$\pi^*(K_X)|_S\equiv (2+\eta_0)C+J_{\eta_0}$$
for some very small rational number $\eta_0>0$. Now the vanishing theorem gives
$$|K_F+\roundup{2L-\frac{1}{2+\eta_0}J_{\eta_0}}|_C=
|K_C+D_{\eta_0}|$$
where $D_{\eta_0}:=\roundup{2L-C-\frac{1}{2+\eta_0}J_{\eta_0}}|_C$ with 
$$\deg(D_{\eta_0})\geq (2-\frac{1}{2+\eta_0})\xi>2. $$
Thus $\map_4|_C$ is birational and so is $\map_4$.
\end{proof} 

{\bf Claim \ref{3-2}.2.}\label{g2} {\em If $K_X^3>144$, $g(C)=2$ and $\xi>1$, then $\map_4$ is birational.}
\begin{proof} One has $(\pi^*(K_X)\cdot S^2)\geq \xi>1$. 
On the other hand, since $\xi\leq (K_F\cdot C)=2$,  
one has $\nu_0\geq \frac{(\pi^*(K_X)^2\cdot S)}{4}>3$ and Lemma \ref{rr} implies
$$\pi^*(K_X)|_S\equiv (3+\eta_1)C+J_{\eta_1}$$ for some rational number $\eta_1>0$ and for some effective $\bQ$-divisor $J_{\eta_1}$. In particular, one has $\beta>3$. 

{}Fix an integer $l_0>5$ such that $\xi\geq \frac{l_0+1}{l_0}$. Set $m'=l_0-1$ and then we have
$$\alpha_{m'}=(l_0-1-2-\frac{1}{\beta})\xi> (l_0-\frac{10}{3})\cdot \frac{l_0+1}{l_0}>l_0-3>1.$$
By Theorem \ref{CZkey}, one gets $\xi\geq \frac{l_0}{l_0-1}$. Recursively running this program as long as $m'\geq 5$, so we eventually get $\xi\geq \frac{6}{5}$. 

Now the vanishing theorem gives
$$|K_F+\roundup{2L-\frac{1}{3+\eta_1}J_{\eta_1}}|_C=
|K_C+D_{\eta_1}|$$
where $D_{\eta_1}:=\roundup{2L-C-\frac{1}{2+\eta_1}J_{\eta_1}}|_C$
with $\deg(D_{\eta_1})\geq (2-\frac{1}{3+\eta_1})\xi>2$. 
Thus $\map_4|_C$ is birational and so is $\map_4$.
\end{proof}
\medskip

\subsection{\bf $\map_4$ in the case $p_g(X)=2,\ 3$ and $d_1=1$}\label{231}\ 
\medskip

{\bf Claim \ref{231}.1.} {\em If $K_X^3>180$, $p_g(X)=3$ and $F$ is not a (1,2) surface, $\map_4$ is birational.}
\begin{proof} We organize the proof by analyzing the numerical types of $F$. 

If $K_{F_0}^2\geq 19$, we have seen $p_g(F)\geq 2$ in the proof of Claim \ref{215}.1. Take $|G|$ to be the moving part of $|K_F|$. Then $\beta\mapsto \frac{2}{3}$ by \cite[Lemma 3.7]{MZ}. Pick a generic irreducible element $C$ of $|G|$. Clearly $\map_4|_F$ distinguishes different general curves $C$. 
When $|G|$ is not composed of a pencil, then $C^2\geq 2$ and $(\sigma^*(K_{F_0})\cdot C)\geq \sqrt{2K_{F_0}^2}\geq \sqrt{38}$. Thus $\xi\geq \frac{2}{3}(\sigma^*(K_{F_0})\cdot C)>2$. Since $\alpha_4\geq (4-1-
\frac{1}{2}-\frac{1}{\beta})\xi>2$, $\map_4$ is birational. 
When $|G|$ is composed of a pencil and $\xi>2$, $\map_4$ is birational for the same reason. Suppose $\xi\leq 2$. Since we have $(\pi^*(K_X)|_F)^2\geq \frac{4}{9}K_{F_0}^2$ and $K_{F_0}^2\geq 19$,  Lemma \ref{rr} implies $\nu_0>2$  and thus $\beta>2$. Noting that $\xi\geq \frac{2}{3}(\sigma^*(K_{F_0})\cdot C)\geq \frac{4}{3}$ since $F$ is not a (1,2) surface, we have $\alpha_4>(4-1-\frac{1}{2}-\frac{1}{2})\xi\geq 2$ and thus $\map_4$ is birational.

If $K_{F_0}^2\leq 18$ and $K_X^3>180$, Lemma \ref{RR} implies
$\pi^*(K_X)\geq_{\text{num}} \frac{10}{3}F$
and $\pi^*(K_X)|_F\geq \frac{10}{13}\sigma^*(K_{F_0})$. 
Take $|G|$ to be the moving part of $|K_F|$ whenever $F$ is a (2,3) surface and, otherwise, $|G|:=|2\sigma^*(K_{F_0})|$. By Lemma \ref{_*}, $\map_4|_F$ distinguishes different generic irreducible elements $C$ of $|G|$. We have $\alpha_4>2$ in the (2,3) case and $\alpha_4>0$ otherwise and thus $\map_4$ is birational by Theorem \ref{CZkey}.
\end{proof}

So we can conclude the following:
\begin{thm}(=Theorem \ref{m4}(2)) Let $X$ be a minimal projective 3-fold of general type. Assume $p_g(X)=3$ and $K_X^3>180$. Then $\map_4$ is birational if and only if $X$ does not contain any genus 2 curve family of canonical degree 1.
\end{thm}
\begin{proof} This is parallel to Theorem \ref{4b} by virtue of Claim \ref{3-2}.1, Claim \ref{3-2}.2 and Claim \ref{231}.1. We omit the details.
\end{proof}

{}Finally we prove the following:

\begin{thm}(=Theorem \ref{m4}(3)) Let $X$ be a minimal projective 3-fold of general type. Assume $p_g(X)=2$ and $K_X^3>855$. Then $\map_4$ is birational if and only if $X$ does not contain any genus 2 curve family of canonical degree 1.
\end{thm}
\begin{proof} Suppose $X$ does not contain any genus 2 curve family of canonical degree 1. We want to show $\map_4$ is birational. We discuss this according to the numerical types of $F$ while we have an induced fibration $f:X'\rw \Gamma$. When $g(\Gamma)>0$, since $F$ is not a (1,2) surface (otherwise, $X$ will have a genus 2 curve family of canonical degree 1), we have known from \cite{MZ} that $\map_4$ is birational. So we may assume $\Gamma\cong \bP^1$. 

Assume $K_{F_0}^2>96$. It is sufficient to show that $|K_F+\roundup{2L}|$ gives a birational map where $L:=\pi^*(K_X)|_F$. We have
$$(2L)^2\geq K_{F_0}^2>8$$
by assumption. Pick two distinct points $P_1$, $P_2$ in very general positions of $F$. If all the curves $C_{1,2}$ passing through $P_1$ and $P_2$ satisfy $(\sigma^*(K_{F_0})\cdot C_{1,2})\geq 4$ (which means $(2L\cdot C_{1,2})\geq 4$), then Lemma \ref{masek} implies that $|K_F+\roundup{2L}|$ separates $P_1$ and $P_2$. Otherwise, there is such a curve $C_{1,2}$ with $(\sigma^*(K_{F_0})\cdot C_{1,2})<4$ and, in fact, these curves $C_{1,2}$ form a curve family (since $P_1$ and $P_2$ are in very general positions). Thus, by our assumption, we have $(L\cdot C_{1,2})>1$. Since $K_{F_0}^2>96$, we have 
$$\nu_0\geq \frac{L^2}{8}\geq \frac{1}{32}
K_{F_0}^2>3.$$
Lemma \ref{rr} implies $\beta_{1,2}>3$. We may estimate the intersection number $\xi_{1,2}:=(L\cdot C_{1,2})$ using similar method to that in the proof of Claim \ref{g2}.2. The result is $\xi_{1,2}\geq \frac{6}{5}$. Now we have $\alpha_{4}^{1,2}\geq (4-1-1-\frac{1}{\beta_{1,2}})\xi>2$ and Theorem \ref{CZkey} implies that $\map_4|_{C_{1,2}}$ is birational and, in particular, $\map_4$ separates $P_1$ and $P_2$ who are in very general positions. Thus  we have seen $\map_4$ is birational. 

Assume $K_{F_0}^2\leq 95$ and $K_X^3>855$. As we have seen before, $F$ can not be a (1,2) surface. Take $|G|$ to be the moving part of $|K_F|$ whenever $F$ is a (2,3) surface and, otherwise, $|G|:=|2\sigma^*(K_{F_0})|$. Lemma \ref{RR} implies $\tau_0>3$ and  
$$\pi^*(K_X)|_F\genum (\frac{3}{4}+\delta_0)\sigma^*(K_{F_0})$$
for some small rational number $\delta_0>0$. By the vanishing theorem, we are in the position to show $|K_F+2\sigma^*(K_{F_0})+\roundup{Q_f}|$ gives a birational map, where $Q_f$ is a nef and big $\bQ$-divisor. In fact, when $F$ is a (2,3) surface, this is the case due to Theorem \ref{CZkey}. When $F$ is neither a (2,3) surface nor a (1,1) surface, $|K_F+ 2\sigma^*(K_{F_0})+\roundup{Q_f}|$ satisfies the condition of Lemma \ref{masek} by a parallel argument to ($\ddag$) in the proof of Claim \ref{R2}.1 and thus it also gives a birational map. {}Finally if $F$ is a (1,1) surface, Lemma \ref{RR} actually gives $\tau_0>200$ and we are in a much better situation. The vanishing theorem again allows us to consider the linear system $|K_F+\roundup{Q_{1,1}}|$ where $Q_{1,1}\genum (2+\frac{199}{201})\sigma^*(K_{F_0})$ and $Q_{1,1}$ is nef and big. Clearly, this linear system also satisfies the condition of Lemma \ref{masek} still by a similar argument to ($\ddag$). In a word, $\map_4$ is birational. 

Conversely, if $X$ has a genus 2 curve family ${\mathscr C}$ of canonical degree 1, $\map_4$ is non-birational. This can be seen by a similar argument to that of \cite[Proposition 4.6]{MZ}. The point is that, while $\pi^*(K_X)\cdot C)=1$, we will be able to see $|4K_{X'}||_C=|2K_C|$ for a general curve $C\in {\mathscr C}$ (This is not a trivial statement at all!). Since $g(C)=2$, $\map_4$ can not be birational. We omit more details and leave it as an exercise.
\end{proof}

\noindent{\bf Acknowledgment.} I would like to thank Max-Planck-Institut f$\ddot{\text{u}}$r Mathematik (Bonn) for the generous support during my visit there in 2011. I was indebted to D.-Q. Zhang for his skillful helps in many parts and for very fruitful discussions. Thanks are also due to Fabrizio Catanese, Gavin Brown, Jungkai A. Chen and Rita Pardini for constant helps and collaborations.


\begin{thebibliography}{99}
\bibitem{AF} M. Abe, M. Furushima, {\em On non-normal del Pezzo surfaces}, Math. Nachr. {\bf 260} (2003), 3--13.

\bibitem{BPV} W. Barth, C. Peters, A. Van de Ven, {\em Compact complex surfaces}, Springer-Verlag, 1984.

\bibitem{Bom} E. Bombieri, {\em Canonical models of surfaces of general type}, Inst. Hautes \'Etudes Sci. Publ. Math. {\bf 42} (1973), 171--219.

\bibitem{C-C} F. Catanese, C. Ciliberto,
{\em Surfaces with $p\sb g=q=1$}. Problems in the theory of surfaces and their classification (Cortona, 1988), 49--79, Sympos.
Math., XXXII, Academic Press, London, 1991.

\bibitem{Ex1} J. A. Chen, M. Chen, {\em Explicit birational geometry of threefolds of general type, I}, Ann. Sci. \'Ec Norm. Sup. {\bf 43} (2010), 365--394. 

\bibitem{Ex2} J. A. Chen, M. Chen, {\em Explicit birational geometry of threefolds of general type, II}, J. of Differential Geometry {\bf 86} (2010), 237--271. 




\bibitem{IJM} M. Chen, {\em Canonical stability of 3-folds of general type with $p_g\geq 3$}, Internat. J. Math. {\bf 14} (2003), no. 5, 515--528. 

\bibitem{MA} M. Chen, {\em A sharp lower bound for the canonical volume of 3-folds of general type}, Math. Ann. 
{\bf 337} (2007), 165--181.

\bibitem{MZ} M. Chen, D.-Q. Zhang, {\em Characterization of the 4-canonical birationality of algebraic threefolds}, Math. Zeit. {\bf 258} (2008), 565--585. 

\bibitem{E-L} L. Ein, R. Lazarsfeld, {\em Global generation of pluricanonical and adjoint linear series on smooth projective threefolds}, J. Amer. Math. Soc. {\bf 6} (1993), no. 4, 875--903.


\bibitem{Flt} A. R. Iano-Fletcher, {\em Working with weighted complete intersections}, Explicit birational geometry of 3-folds, 101--173, London Math. Soc. Lecture Note Ser., {\bf 281}, Cambridge Univ. Press, Cambridge, 2000.
\bibitem{Fr} P. Francia,
 {\em On the base points of the bicanonical system}.
Problems in the theory of surfaces and their classification (Cortona, 1988), 141--150, Sympos. Math., XXXII, Academic Press, London, 1991.

\bibitem{KV} Y. Kawamata, {\em A generalization of Kodaira-Ramanujam's vanishing theorem}, Math. Ann. {\bf 261} (1982), 43--46.



\bibitem{Masek} V. Masek,{\em Very ampleness of adjoint linear systems on smooth surfaces with boundary}, Nagoya Math. J. {\bf 153} (1999), 1--29.


\bibitem{MZ1}  M. Miyanishi, D.-Q. Zhang, {\em Gorenstein log del Pezzo surfaces of rank one}, J. Algebra {\bf 118} (1988), no. 1, 63--84. 


\bibitem{MZ2} M. Miyanishi, D.-Q. Zhang, {\em Gorenstein log del Pezzo surfaces, II}, J. Algebra {\bf 156} (1993), no. 1, 183--193


\bibitem{Miy} Y. Miyaoka, {\em Tricanonical maps of numerical Godeaux surfaces}, Invent. Math. {\bf 34} (1976), no. 2, 99Ð-111. 

\bibitem{RN} M. Reid, {\em Nonnormal del Pezzo surfaces}, Publ. Res. Inst. Math. Sci. Kyoto Univ. {\bf 30} (1994), 695--727.

\bibitem{YPG} M. Reid, {\em Young person's guide to canonical singularities}. Algebraic geometry, Bowdoin, 1985 (Brunswick, Maine, 1985), 345Ð-414, 
Proc. Sympos. Pure Math., {\bf 46}, Part 1, Amer. Math. Soc., Providence, RI, 1987.
\bibitem{Reid} M. Reid, {\em Chapters on algebraic surfaces}, Complex algebraic geometry (Park City, UT, 1993), 3Ð-159, IAS/Park City Math. Ser., {\bf 3}, Amer. Math. Soc., Providence, RI, 1997. 

\bibitem{Reider} I. Reider, {\em Vector bundles of rank 2 and linear systems on
algebraic surfaces}, Ann. Math. {\bf 127}(1988), 309-316.



\bibitem{VV} E. Viehweg, {\em Vanishing theorems}, J. reine angew. Math. {\bf 335} (1982), 1--8.

\bibitem{Ye} Q. Ye, {\em On Gorenstein log del Pezzo surfaces}, Japan. J. Math. {\bf 28} (2002), no. 1, 87--136.


\bibitem{Zhou} Y. Zhou, {\em The 3-canonical system on 3-folds of general type}. Osaka J. Math. {\em 48} (2011), no. 1, 91--98.
\bibitem{Zhu} L. Zhu, {\em The generic finiteness of the m-canonical map for 3-folds of general type}, Osaka J. Math. {\bf 42} (2005), no. 4, 873--884.
\end{thebibliography}
\end{document}